\documentclass[11pt]{amsart}
\usepackage[T1]{fontenc}
\usepackage{amsfonts}
\makeatletter
\theoremstyle{plain}
\newtheorem{thm}{Theorem}[section]
\theoremstyle{plain}
\numberwithin{equation}{section}
\numberwithin{figure}{section} 
\theoremstyle{plain}
\newtheorem{cor}[thm]{Corollary} 
\theoremstyle{plain}
\newtheorem{lem}[thm]{Lemma} 
\theoremstyle{plain}
\newtheorem{pr}[thm]{Proposition} 
\theoremstyle{plain}
\newtheorem{rem}[thm]{Remark}
\theoremstyle{plain}
\newtheorem{definition}[thm]{Definition} 

\begin{document}

\title{Riesz multiplier convergent spaces of operator valued series and a version of Orlicz- Pettis theorem}
\author{Mahmut Karaku\c s 
and Ramazan Kama}
\address[Mahmut Karaku\c s]{Van Y\"uz\"unc\"u Y{\i}l University, Faculty of Science, Department of Mathematics, 65100 - Van, Turkey}
\email{matfonks@gmail.com, mkarakus@yyu.edu.tr}
\address[Ramazan Kama] {S\.{i}\.{i}rt University, Faculty of Education, Department of Mathematics and Physical Sciences Education, 56100 - S\.{i}\.{i}rt, Turkey}
\email{ramazankama@siirt.edu.tr}

\begin{abstract}
It is not usual to characterize an operator valued series via completeness of multiplier spaces. In this study, by using a series of bounded linear operators, we introduce the space $M^\infty_{R}\big(\sum_k T_k\big)$ of Riesz summability which is a generalization of the Ces\`{a}ro summability. Therefore, we give the completeness criteria of these spaces with $c_0(X)$-multiplier convergent operator series. It is a natural consequence that one can characterize the completeness of a normed space through $M^\infty_{R}\big(\sum_k T_k\big)$ which will be assumed that is complete for every $c_0(X)$-multiplier Cauchy operator series. Then, we characterize the continuity and the (weakly) compactness of the summing operator $\mathcal{S}$ from the multiplier space $M^\infty_{R}\big(\sum_k T_k\big)$ to an arbitrary normed space $Y$ through $c_0(X)$-multiplier Cauchy and $\ell_\infty(X)$-multiplier convergent series, respectively. We also prove that if $\sum_kT_k$ is $\ell_\infty(X)$-multiplier Cauchy, then the multiplier space of weakly Riesz-convergence associated to the operator valued series $M^\infty_{wR}\big(\sum_k T_k\big)$ is subspace of $M^\infty_{R}\big(\sum_k T_k\big)$. Among other results, finally, we obtain a new version of the well-known Orlicz-Pettis theorem by using Riesz-summability.
\end{abstract}
\maketitle
\section{Introduction}
Let $\mathbb{N}$ be the set of positive integers, $\mathbb{R}$ and $\mathbb{C}$ also be the real and complex fields as usual, respectively. By $\omega$, we denote the space of all real (or complex) valued sequences and any vector subspace of $\omega$ is also called as a \textit{sequence space}. A $K$ space is a locally convex sequence space $X$ containing $\phi$ on which coordinate functionals $\pi_{k}(x)=x_{k}$ are continuous for every $k\in\mathbb{N}$, the set of positive integers, where $\phi$ is the space of finitely non-zero sequences spanned by the set $\{e^{k}:k\in\mathbb{N}\}$. $e^{k}$ is the sequence whose only non-zero term is 1 in the $k^{th}$ place for all $k\in\mathbb{N}$, and $e$ is also the sequence with $e=(1,1,...)$. A complete linear metric (or complete normed) $K$ space is called an $FK$ (or a $BK$) space. Let $X\supset\phi$ be a $BK$ space and $x=(x_{k})\in X$. Then, by $x^{[n]}=\sum_{k=1}^{n}x_{k}e^{k}$ for all $n\in\mathbb{N}$, we denote the $n^{th}$ section of $x$. It is said that $X\supset\phi$ is an $AK$ space if $\big\|x^{[n]}-x\big\|_X\to 0$, as $n\to\infty$, for each $x=(x_{k})\in X$. The spaces $\ell_\infty$, $c$ and $c_0$ of bounded, convergent and null sequences are $BK$ spaces, respectively, with the sup norm $\|x\|_{\infty}=\sup_{k\in\mathbb{N}}|x_{k}|$.

Let $X$, $Y$ be any two sequence spaces and $A=(a_{nk})$ be an infinite matrix of complex numbers $a_{nk}$, where $k,n\in\mathbb{N}$.  Then, we say that $A$ defines a \textit{matrix transformation} from $X$ into $Y$ and we denote it by writing $A:X\to Y,$ if for every sequence $x=(x_{k})\in X$ the sequence $Ax=\{(Ax)_{n}\}$, the $A$-transform of $x$, is in $Y$; where
\begin{eqnarray}\label{eq11}
(Ax)_{n}=\sum_{k}a_{nk}x_{k}.
\end{eqnarray}For simplicity in notation, here and after, the summation without limits runs from 1 to $\infty$. By $(X:Y)$, we denote the class of all matrices $A$ such that $A:X\to Y$. Thus, $A\in(X:Y)$ if and only if the series on the right side of (\ref{eq11}) converges for
each $n\in\mathbb{N}$. Furthermore, the sequence $x$ is said to be $A$-summable to $a\in\mathbb{C}$ if $Ax$ converges
to $a$ which is called the $A$-limit of $x$. The reader can refer to \cite{basar1, boos1} and \cite{wila3} for recent results and related topics in summability.

We say that the series $\sum_{k}x_{k}$ in a normed space $X$ is unconditionally convergent ($uc$) or unconditionally Cauchy ($uC$) if the series
$\sum_{k}x_{\pi(k)}$ converges or is a Cauchy series for every permutation $\pi$ of $\mathbb{N}$. It is called weakly unconditionally Cauchy ($wuC$) if for every permutation $\pi$ of $\mathbb{N}$, the sequence $\big(\sum_{k=1}^nx_{\pi(k)}\big)$ is a weakly Cauchy sequence or as a useful result, $\sum_{k}x_{k}$ is $wuC$ if and only if $\sum_{k}|x^*(x_{k})|<\infty$ for all $x^*\in X^*$, the space of all linear and bounded (continuous) functionals defined on $X$. It is well known that every $wuC$ series in a Banach space $X$ is $uc$ if and only if $X$ contains no copy of $c_0$; (cf. \cite[pp. 42, 44]{kalton}, \cite[p. 44]{dies} and \cite[p. 18]{jurg}). The reader can refer to the Diestel' s famous monograph \cite{dies} given on the theory of sequences and series in Banach spaces, Albiac and Kalton \cite{kalton} for specific investigations of Banach spaces and Marti's  \cite{jurg} for basic sequences and fundamentals of bases.

Let $X$ and $Y$ be two normed spaces and $\omega(X)$ be the space of all $X$-valued sequences. By $\ell_\infty(X)$, $c(X)$ and $c_0(X)$, we denote the spaces of all $X$-valued bounded, convergent and null sequences, respectively. $\phi(X)$ also denotes the space of $X$-valued finitely non-zero sequences. Let $\mathcal{V}$ be a vector space of $X$-valued sequences equipped with a locally convex Hausdorff topology. If the maps $x=(x_k)\mapsto x_k$ from $\mathcal{V}$ into $X$ are continuous for all $k\in\mathbb{N}$, then $\mathcal{V}$ is called as a $K$ space. If $x\in X$, then by $e^{k}\otimes x$, we denote the sequence whose only non-zero term is $x$ in the $k^{th}$ place for all $k\in\mathbb{N}$. If $\phi(X)\subset\mathcal{V}$ and $T_k\in B(X:Y)$, the space of all bounded and linear operators defined from $X$ into $Y$ for all $k\in\mathbb{N}$, then we say that the series $\sum_kT_k$ is $\mathcal{V}$-multiplier convergent or $\mathcal{V}$-multiplier Cauchy if the series $\sum_k T_kx_k$ converges in $Y$ or is a Cauchy series in $Y$, i.e., the partial sums of the series $\sum_k T_k x_k$ form a norm Cauchy sequence in $Y$ for all $x=(x_k)\in\mathcal{V}$. The reader may refer to \cite{s2} for the vector valued multiplier spaces associated with operator valued series (OVS) and more detailed information on multiplier convergent series.

The multiplier form of a series $\sum_kx_k$ in a Banach space $X$ associated with an arbitrary real or complex sequence $a=(a_k)$ is given as $\sum_ka_kx_k$ and is also important to understand the behavior of the series $\sum_kx_k$ in $X$. The series $\sum_kx_k$ in $X$ is $wuC$ (or $uc$) series if and only if $\sum_ka_kx_k$ is convergent for every null (or bounded) sequence $a=(a_k)$, that is, $\sum_kx_k$ is a $c_0$-(or an $\ell_\infty$-) multiplier convergent series. As another example, a series $\sum_kx_k$ is subseries convergent if and only if it is $m_0$-multiplier convergent series, where $m_0$ is the space of all finite range sequences. Let also recall that a series $\sum_{k}x_{k}$ is subseries convergent if and only if $\sum_{k}x_{n_{k}}$ is convergent for all subsequences $(n_{k})$ of $\mathbb{N}$. There is also an important result on subseries convergence which states that in a normed space $X$, a subseries convergent series is $\ell_\infty$-multiplier Cauchy, if also $X$ is sequentially complete then the series is $\ell_\infty$-multiplier convergent.

The studies of characterizations of Banach spaces, and obtaining new multiplier spaces by means of  summability methods have been an interesting field of research in the theory of modern analysis. In \cite{McArthur}, McArthur investigated some closed linear subspaces of the space of the weakly unconditional summable sequences in a Banach space $X$ and gave some inclusion relations between these spaces. One can refer to the Aizpuru et al. \cite{a1,a4}, P\'{e}rez-Fern\'{a}ndez et al. \cite{a5},
Kama and Altay \cite{al3}, Kama et al. \cite{al4}, Kama \cite{kama,kama1}, Karaku\c s \cite{karakus1}, Karaku\c s and Ba\c sar \cite{fbmk1, fbmk2,fbmk4} and Swartz \cite{s2,s1} for recent studies on multiplier convergent series.

Through the rest of the paper, we deal with the Riesz transformation which is a generalization of Ces\`{a}ro mean $C_{1}$ of order one given as
\begin{eqnarray}\label{ces}
\begin{array}{clccl}
C_{1} & : &\omega&\longrightarrow &\omega\\
          &    & x=(x_{k})&\longmapsto &C_{1}x=\left(\frac{1}{n}\sum_{k=1}^nx_k\right)_{n\in\mathbb{N}}.
\end{array}
\end{eqnarray}
A series $\sum_kx_k$ in a normed space $X$ is said to be \textit{Ces\`{a}ro} or \textit{weak Ces\`{a}ro convergent to $x_0\in X$} and is denoted by $C_1-\sum_kx_k=x_0$ or $wC_1-\sum_kx_k=x_0$, if
\begin{eqnarray*}
\lim_{n\to\infty}\left[\frac{1}{n}\sum_{k=1}^n(n-k+1)x_k\right]=x_0
\end{eqnarray*}
or
\begin{eqnarray*}
\lim_{n\to\infty}\left[\frac{1}{n}\sum_{k=1}^n(n-k+1)x^*(x_k)\right]=x^*(x_0)
\end{eqnarray*}
for all $x^*\in X^*$. Riesz transformation is also given by
\begin{eqnarray*}\label{riesz}
Rx=\left(\frac{1}{R_n}\sum_{k=1}^nr_kx_k\right)_{n\in\mathbb{N}}
\end{eqnarray*}
for any sequence $x=(x_k)_{k\in\mathbb{N}}\in\omega$, where $r=(r_k)$ is a sequence of nonnegative reals with $r_1>0$ such that
\begin{eqnarray}\label{rieszc} R_n=\sum_{k=1}^nr_k~\textrm{ with }~\lim_{n\to\infty}R_{n}=\infty.
 \end{eqnarray}
The Riesz method is regular under the condition (\ref{rieszc}) and is reduced for $r=e$ to the Ces\`{a}ro mean of order one. The corresponding matrix $R=(r_{nk})$ to the Riesz method with respect to the sequence $r=(r_{k})$ can be given, as follows;
\begin{eqnarray*}
r_{nk}:=\left\{\begin{array}{ccl}
\frac{r_{k}}{R_{n}} &,& 1\leq k\leq n,\\
0&,& k>n
\end{array} \right.
\end{eqnarray*}
for all $k,n\in\mathbb{N}$.

In this study, our main purpose is to give some results related to the some new multiplier by using Riesz summability method as a generalization of Ces\`{a}ro summability, and construct some new classes of sequence spaces associated to an OVS in a Banach space having similar properties with the classes defined in \cite{al2,karakus1,fbmk2,fbmk4} and \cite{s1}. Among other results, we give some new characterizations of $c_0(X)$- and $\ell_\infty(X)$-multiplier convergent series through this method.

The following lemma states a classical result of the $wuC$ series in a normed space $X$, and gives an example of multiplier convergent series which characterizes  unconditionally convergent series and weakly unconditionally Cauchy series with $\ell_\infty$-multiplier convergent series and $c_0$-multiplier convergent series, respectively; \cite[p. 44]{dies}.
\begin{lem}\label{l1}
The following statements hold:
\begin{enumerate}
\item[(a)] In a normed space $X$, a formal series $\sum_nx_n$ is a $wuC$ series if and only if there exists a positive real $H$ such that
\begin{eqnarray*}\label{s}
H=\sup_{n\in\mathbb{N}}\left\{\left\|\sum_{k=1}^na_{k}x_{k}\right\|:|a_{k}|\leq1,~k=1, 2,\ldots ,n \right\}.
 \end{eqnarray*}
\item[(b)] A formal series $\sum_nx_n$ in a
Banach space $X$ is $uc$ (respectively $wuC$) series if and only if for any $(t_n)\in \ell_\infty$ (respectively for any $(t_n)\in c_0$), $\sum_nt_nx_n$ converges, that is, $\sum_nx_n$ is an $\ell_\infty$-(respectively a $c_0$-) multiplier convergent series.
\end{enumerate}
\end{lem}
\section{Vector Valued Multiplier Spaces Through Riesz Summability and compact summing operator}
In this section, we give some definitions on Riesz convergent or Riesz summable sequences in a real normed space.
\begin{definition}\label{def11}
A sequence $x=(x_{k})$ in a real normed space $X$ is said to be Riesz convergent ($R$-convergent) and weakly Riesz convergent ($wR$-convergent) to $x_0\in X$ which is called the $R$-limit and $wR$-limit of $x$, and is denoted by $R-\lim_{k\to\infty} x_{k}=x_0$ and $wR-\lim_{k\to\infty}x_{k}=x_0$, if
\begin{eqnarray*}
\lim_{n\to\infty}\left\|\frac{1}
{R_n}\sum_{k=1}^{n}r_kx_{k}-x_0\right\|=0
\end{eqnarray*}
and
\begin{eqnarray*}
\lim_{n\to\infty}\left|\frac{1}
{R_n}\sum_{k=1}^{n}r_kx^*(x_{k})-x^*(x_0)\right|=0~\textrm{ for all }~x^*\in X^*
\end{eqnarray*}
hold, respectively.
\end{definition}
Riesz summability of a sequence $x=(x_{k})$ (or $R$- and $wR$-convergence of a series) in a real normed space $X$ is also given, as follows:
\begin{definition}\label{def12}
A series $\sum_kx_{k}$ in a real normed space $X$ is said to be Riesz convergent ($R$-convergent) and weakly Riesz convergent ($wR$-convergent) to $x_0\in X$ which is called the $R$-sum and $wR$-sum of $x$, and is denoted by $R-\sum_k x_{k}=x_0$ and $wR-\sum_kx_{k}=x_0$, if
\begin{eqnarray*}
\lim_{n\to\infty}\left\|\frac{1}
{R_n}\sum_{k=1}^{n}r_ks_{k}-x_0\right\|=0
\end{eqnarray*}
and
\begin{eqnarray*}
\lim_{n\to\infty}\left|\frac{1}
{R_n}\sum_{k=1}^{n}r_kx^*(s_{k})-x^*(x_0)\right|=0~\textrm{ for all }~x^*\in X^*
\end{eqnarray*}
hold, respectively. Here, $s_{k}=\sum_{j=1}^{k}x_j$ is the sequence of partial sums of the series $\sum_kx_{k}$.
\end{definition}

In the following, we introduce the spaces of $R$-convergence and weakly $R$-convergence of multipliers with association of an OVS, and obtain some new characterizations of $c_0(X)$- and $\ell_\infty(X)$-multiplier convergent (Cauchy) series. We also present the necessary and sufficient conditions for the continuity and (weak) compactness of the summing operator $\mathcal{S}$ defined from these multiplier spaces to another normed space $Y$, and derive some results by the previous works \cite{al2,fbmk2} and \cite{s1}.

\begin{definition}
 Suppose that $X$ and $Y$ are two normed spaces and $T_k\in B(X:Y)$ for all $k\in\mathbb{N}$. We define the space $M^\infty_{R}\big(\sum_k T_k\big)$ of $R$-convergence associated to the $\sum_k T_k$ as;
\begin{eqnarray}\label{main}
M^\infty_{R}\big(\sum_k T_k\big):=\left\{x=(x_k)\in\ell_\infty(X):\sum_k T_k x_{k}~ \textrm{is }R\textrm{-convergent}\right\},
\end{eqnarray}
and endowed with the sup norm.
 \end{definition}
It can be easily checked that the inclusions
\begin{eqnarray}\label{inc}
\phi(X) \subseteq M^\infty_{R}\big(\sum_k T_k\big)\subseteq\ell_\infty(X)
\end{eqnarray}
hold.

Firstly, we give the following theorem related to the completeness of the normed space $M^\infty_{R}\big(\sum_k T_k\big)$ through $c_0(X)$-multiplier convergence.
\begin{thm}\label{t1} Regarding two Banach spaces $X$ and $Y$, and a series  $\sum_k T_k$, where $T_k\in B(X:Y)$ for all $k\in\mathbb{N}$, the following assertions are equivalent:
\begin{enumerate}
\item[(i)] The series $\sum_k T_k$ is $c_0(X)$-multiplier convergent.
\item[(ii)] The space $M^\infty_{R}\big(\sum_k T_k\big)$ is a Banach space.
\end{enumerate}
\end{thm}
\begin{proof} (i)$\Rightarrow$(ii):
Let us suppose that the series $\sum_k T_k$ is $c_0(X)$-multiplier convergent. Then, from Part (a) of Lemma \ref{l1} there exists $H>0$ such that
\begin{eqnarray*}
H=\sup_{n\in\mathbb{N}}\left\{\left\|\sum_{k=1 }^nT_kx_k\right\|:\|x_k\|\leq 1~\textrm{ for all }~k\in\{1,2,\ldots,n\}\right\}.
\end{eqnarray*}
Now suppose that $(x^m)=(x_k^m)_{k\in\mathbb{N}}$ is a Cauchy sequence in $M^\infty_{R}\big(\sum_k T_k\big)$. Then, since $\ell_\infty(X)$ is a Banach space and the inclusion relation (\ref{inc}) holds, one can find $x^0=(x_k^0)\in \ell_\infty(X)$ such that $x^m\to x^0$, as $m\to\infty$. Now, we shall prove that $x^0\in M^\infty_{R}\big(\sum_k T_k\big)$. For an arbitrary $\epsilon>0$, we have $m_0\in\mathbb{N}$ satisfying
\begin{eqnarray*}
\|x^m-x^0\|<\frac{\epsilon}{3H}
\end{eqnarray*}
for every $m>m_0$. Now, it is immediate by easy calculation that
\begin{eqnarray*}
\frac{3H}{\epsilon}\left\|\frac{1}{R_n}\sum_{k=1}^{n} r_k\sum_{j=1}^kT_j(x_j^m-x_j^0)\right\|\leq H
\end{eqnarray*}
for every $m>m_0$ and $k\in\mathbb{N}$. Therefore, for every $\epsilon>0$ there exists $m_0\in\mathbb{N}$ such that
\begin{eqnarray*}
\left\|\frac{1}{R_n}\sum_{k=1}^{n} r_k\sum_{j=1}^kT_j(x_j^m-x_j^0)\right\|<\frac{\epsilon}{3}
\end{eqnarray*}
for all $m>m_0$ and $k\in\mathbb{N}$. Since $(x^m)$ is a Cauchy sequence in $M^\infty_{R}\big(\sum_k T_k\big)$, we have a sequence $(y_m)$ in the Banach space $Y$ for which the following inequality
\begin{eqnarray*}
\left\|\frac{1}{R_n}\sum_{k=1}^{n} r_k\sum_{j=1}^kT_j(x_j^m-y_m)\right\|<\frac{\epsilon}{3}
\end{eqnarray*}
holds for every $n>n_0$. Then, since $Y$ is a Banach space and $(y_m)$ is a Cauchy sequence we can find $y_0\in Y$ such that $y_m\to y_0$, as $m\to\infty$, and so for every $\epsilon>0$,
\begin{eqnarray*}
\|y_m-y_0\|<\frac{\epsilon}{3}
\end{eqnarray*}
holds. Finally, for every $\epsilon>0$ there exists $n_0\in\mathbb{N}$ such that
\begin{eqnarray*}
 \left\|\frac{1}{R_n}\sum_{k=1}^{n} r_k\sum_{j=1}^kT_jx_j^0-y_0\right\|&\leq&\left\|\frac{1}{R_n}\sum_{k=1}^{n} r_k\sum_{j=1}^kT_j(x_j^0-x_j^m)\right\|+\\
 &+&\left\|\frac{1}{R_n}\sum_{k=1}^{n} r_k\sum_{j=1}^kT_jx_j^m-y_m\right\|+\|y_m-y_0\|\\
 &<&\epsilon
 \end{eqnarray*}
holds for every $n>n_0$. Hence, $ x^0\in M^\infty_{R}\big(\sum_{k}x_{k}\big)$.

(ii)$\Rightarrow$(i):
Let us suppose that the multiplier space $M^\infty_{R}\big(\sum_k T_k\big)$ is complete and let us take $x=(x_k)\in c_0(X)$. Then, we have $c_0(X)\subseteq M^\infty_{R}(\sum_k T_k)$ since the space $M^\infty_{R}(\sum_k T_k)$ is closed and the inclusion $\phi(X)\subset M^\infty_{R}(\sum_k T_k)$ holds. Therefore, the series $\sum_kT_kx_k$ is $R$-convergent for all $x=(x_k)\in c_0(X)$. From the monotonicity of $c_0(X)$, the series $\sum_{k}T_{k}x_{k}$ is subseries $R$-convergent, and so is weakly subseries $R$-convergent. From a version of Orlicz-Pettis theorem for regular matrices \cite[Theorem 4.1]{a4}, $\sum_kT_kx_k$ is subseries norm convergent, that is, the series $\sum_k T_k$ is $\ell_\infty(X)$-multiplier convergent, and so is $c_0(X)$-multiplier convergent.
 \end{proof}
 \begin{cor}\label{qw}
Let $X$ and $Y$ be any given two Banach spaces and a series  $\sum_k T_k$ with $T_k\in B(X:Y)$ for all $k\in\mathbb{N}$. Then, the following statements are equivalent:
\begin{enumerate}
\item[(i)] The series $\sum_k T_k$ is $c_0(X)$-multiplier convergent.
\item[(ii)] The inclusion $c_0(X)\subseteq M_R^\infty\big(\sum_k T_k\big)$ holds.
\end{enumerate}
 \end{cor}
 \begin{rem} Let $X$ and $Y$ be any given two Banach spaces, and $T_k\in B(X:Y)$ for all $k\in\mathbb{N}$. The multiplier spaces $M^\infty\big(\sum_k T_k\big)$, $M^\infty_C\big(\sum_k T_k\big)$ and $M^\infty_f\big(\sum_k T_k\big)$ are introduced in \cite{s1}, \cite{al2} and \cite{fbmk2}, respectively, as follows;
 \begin{eqnarray}\label{smain}
 M^\infty\big(\sum_k T_k\big)&:=&\left\{x=(x_k)\in\ell_\infty(X):\sum_kT_kx_k \textrm{ is  convergent}\right\},\\
\label{cmain}
M^\infty_C\big(\sum_k T_k\big)&:=&\left\{x=(x_k)\in\ell_\infty(X):\sum_kT_kx_k \textrm{ is  Ces\`{a}ro convergent}\right\},\\
\label{fmain}
M^\infty_f\big(\sum_k T_k\big)&:=&\left\{x=(x_k)\in\ell_\infty(X):\sum_kT_kx_k \textrm{ is almost convergent}\right\}.
\end{eqnarray}
Now, by considering the definitions of the multiplier spaces  $M^\infty_{R}\big(\sum_k T_k\big)$, $M^\infty\big(\sum_k T_k\big)$, $M^\infty_C\big(\sum_k T_k\big)$ and $M^\infty_{f}\big(\sum_k T_k\big)$ respectively given by (\ref{main}), (\ref{smain}), (\ref{cmain})and (\ref{fmain}), we have the following inclusions:
\begin{eqnarray*} M^\infty\big(\sum_k T_k\big)\subseteq M_f^\infty\big(\sum_k T_k\big)\subseteq M_C^\infty( \sum_k T_k)\subseteq M_R^\infty( \sum_k T_k).
\end{eqnarray*}
\end{rem}
 By bearing in mind Theorem \ref{t1} together with previous results given in \cite{al2} and \cite{fbmk2}, we have the following:
\begin{cor}\label{cor1} Let $X$ and $Y$ be Banach spaces, and $T_k\in B(X:Y)$ for all $k\in\mathbb{N}$. Then, the following assertions are equivalent:
\begin{enumerate}
\item[(i)] The series $\sum_k T_k$ is $c_0(X)$-multiplier convergent.
\item[(ii)] $M^\infty_{f}\big(\sum_k T_k\big)$ is a Banach space.
\item[(iii)] $M^\infty_C\big(\sum_k T_k\big)$ is a Banach space.
    \item[(iv)] $M^\infty_R\big(\sum_k T_k\big)$ is a Banach space.
\end{enumerate}
\end{cor}
\begin{rem}
Let $X$ and $Y$ be normed spaces and $T_k\in B(X:Y)$ for all $k\in\mathbb{N}$. Consider the vector valued multiplier Cauchy space $CM^\infty\big(\sum_k T_k\big)$ and the vector valued multiplier Riesz-Cauchy space $CM_R^\infty\big(\sum_k T_k\big)$ which are associated to an OVS, and defined by
\begin{eqnarray}\label{cauchy}
CM^\infty\big(\sum_k T_k\big):=\left\{x=(x_k)\in\ell_\infty(X):\sum_kT_kx_k \textrm{ is a Cauchy series}\right\}
\end{eqnarray}
and
\begin{eqnarray}\label{rcauchy}
CM_R^\infty\big(\sum_k T_k\big):=\left\{x=(x_k)\in\ell_\infty(X):\sum_kT_kx_k \textrm{ is a Riesz-Cauchy series}\right\}.
\end{eqnarray}
\end{rem}
By using the spaces given in (\ref{cauchy}) and (\ref{rcauchy}), we have the following corollary as an analogue of corresponding result given in \cite[Remark 2.4]{a1}.
\begin{cor}\label{cmf} Let $X$ and $Y$ be normed spaces and $T_k\in B(X:Y)$ for all $k\in\mathbb{N}$. The following statements are equivalent:
\begin{enumerate}
\item[(i)] $\sum_k T_k$ is $c_0(X)$-multiplier convergent.
\item[(ii)] $CM^\infty\big(\sum_k T_k\big)$ is a Banach space.
\item[(iii)] $CM_R^\infty\big(\sum_k T_k\big)$ is a Banach space.
\end{enumerate}
\end{cor}
 The following theorem is an analogue of Theorem 2.5 given in \cite{al2}, and characterizes the completeness of a normed space by the vector valued multiplier space $M^\infty_{R}\big(\sum_k T_k\big)$. Since the proof is similar to the case $M^\infty_C\big(\sum_k T_k\big)$ given in \cite{al2}, we omit details; (see also \cite[Corollary 1.8]{s1}).
\begin{thm}\label{y} Let $X$ and $Y$ be given any normed spaces such that $X$ is complete and $T_k\in B(X:Y)$ for all $k\in\mathbb{N}$. Then, the space $Y$ is also complete if and only if $M^\infty_{R}\big(\sum_k T_k\big)$ is complete for every $c_0(X)$-multiplier Cauchy series.
\end{thm}
In the following theorem, we give a new characterization of continuity of the certain summing operator by using Riesz summability method and $c_0(X)$-multiplier Cauchy series.
\begin{thm}\label{t3} Let $X$ and $Y$ be any normed spaces and $T_k\in B(X:Y)$ for all $k\in\mathbb{N}$. Then, the summing  operator $\mathcal{S}$ defined by
\begin{eqnarray}\label{con}
\begin{array}{ccccl}
\mathcal{S}& : &M^\infty_{R}\big(\sum_k T_k\big)&\longrightarrow &Y \\
          &    & x=(x_k)&\longmapsto &\mathcal{S}x=R-\sum_kT_kx_{k}
\end{array}
\end{eqnarray}
is continuous if and only if the series $\sum_{k}T_{k}$ is $c_0(X)$-multiplier Cauchy.
\end{thm}
\begin{proof} Let us suppose that $ \mathcal{S}$ is continuous and define the set $G$ by
\begin{eqnarray}\label{eq213}
G:=\left\{\left\|\sum_{k=1}^nT_kx_k\right\|:\|x_k\|\leq 1~\textrm{ for all }~k\in\{1,2,\ldots,n\}\right\}.
\end{eqnarray}
Since the inclusion $\phi(X)\subset M^\infty_{R}(\sum_{k}T_{k})$ holds, the series $\sum_{k}T_{k}$ is $c_0(X)$-multiplier Cauchy from  the inequality $H=\sup_{n\in\mathbb{N}}G\leq \|\mathcal{S}\|$.

Conversely, let us suppose that $\sum_{k}T_{k}$ is $c_0(X)$-multiplier Cauchy series. Therefore, the set $G$ defined by (\ref{eq213}) is bounded (see \cite[Theorem 1.3]{s1}) and so, $H=\sup_{n\in\mathbb{N}}G$.
If $x=(x_{k})\in M^\infty_{R}\big(\sum_k T_k\big)$, then the proof follows from the inequality
\begin{eqnarray*}
\|\mathcal{S}x\|=\left\|R-\sum_k T_{k}x_{k}\right\|\leq H\|x\|.
\end{eqnarray*}
\end{proof}
Combining Theorem \ref{t3} with the results due to Altay and Kama \cite{al2} and Karaku\c s and Ba\c sar \cite{fbmk2}; we have the following:
\begin{cor}\label{c5}
Let $X$, $Y$ be any normed spaces and $T_k\in B(X:Y)$ for all $k\in\mathbb{N}$. Then, the following statements are equivalent:
\begin{enumerate}
\item[(i)]  The series $\sum_k T_k$ is $c_0(X)$-multiplier Cauchy.
\item[(ii)]  $\mathcal{S}:M_f^\infty\big(\sum_k T_k\big)\to Y$ is continuous.
\item[(iii)] $\mathcal{S}:M^\infty_C\big(\sum_k T_k\big)\to Y$ is continuous.
\item[(iv)] $\mathcal{S}:M^\infty_{R}\big(\sum_k T_k\big)\to Y$ is continuous.
    \end{enumerate}
\end{cor}
Now, we give the characterization of the (weakly) compactness for the summing operator $\mathcal{S}$ by using Riesz summability method and  $\ell_\infty(X)$-multiplier convergent series. Let us note that  $X$ need not to be complete.
\begin{thm}\label{lt} Let $X$ be any normed space, $Y$ be a Banach space and $T_k\in B(X:Y)$ for all $k\in\mathbb{N}$. Then the series $\sum_{k}T_{k}$ is $\ell_\infty(X)$-multiplier convergent if and only if the summing  operator $\mathcal{S}$ defined by (\ref{con}) is compact (weakly compact).
\end{thm}
\begin{proof}
Let us suppose that $\mathcal{S}$ is compact. If $x=(x_k)\in\ell_\infty(X)$, then the set
\begin{eqnarray*}
H:=\left\{\sum_{k\in\mathcal{F}}e^k\otimes x_k|\mathcal{F} \textrm{ finite and }\|x_k\|\leq1\right\}\subset M^\infty_{R}\big(\sum_k T_k\big)
\end{eqnarray*}
is bounded. By the hypothesis,
\begin{eqnarray*}
\mathcal{S}(H):=\left\{R-\sum_{k\in\mathcal{F}}T_k x_k|\mathcal{F} \textrm{ finite and }\|x_k\|\leq1\right\}
\end{eqnarray*}
is relatively compact. Therefore, the series $\sum_kT_kx_k$ is subseries norm $R$-convergent, and so is weakly subseries $R$-convergent \cite[Theorem 2.48]{s2}. Further, by a consequence of Orlicz-Pettis theorem for regular matrices \cite{a4}, the series  $\sum_kT_kx_k$ is  subseries norm convergent; that is the series $\sum_k T_k$ is  $\ell_\infty(X)$-multiplier convergent.

Conversely, suppose that the series $\sum_k T_k$ is $\ell_\infty(X)$-multiplier convergent. We define the operators $\mathcal S^R_n$ by
\begin{eqnarray*}
\begin{array}{ccccl}
\mathcal S^R_n& : &M^\infty_{R}\big(\sum_k T_k\big)&\longrightarrow &Y \\
          &    & x=(x_k)&\longmapsto &\mathcal{S}^R_n(x)=R-\sum_{k=1}^nT_kx_{k}
\end{array}
\end{eqnarray*}
for all $n\in\mathbb{N}$. It is sufficient to prove that $\|\mathcal S^R_n-\mathcal S\|\to0$, as $n\to\infty$. Since  $\sum_k T_k$ is  $\ell_\infty(X)$-multiplier convergent, then the series $\sum_kT_kx_k$ is uniformly $R$-convergent for $\|x_k\|\leq1$ \cite[Corollary 11.11]{s2}. Therefore,
\begin{eqnarray*}
\lim_{n\to\infty}\left\|\mathcal S^R_n-\mathcal S\right\|&=&\lim_{n\to\infty}\left\|\left(R-\sum_{k=1}^nT_kx_k\right)-
\left(R-\sum_{k} T_kx_k\right)\right\|\\
&=&\lim_{n\to\infty}\left\|R-\sum_{k=n+1}^\infty T_kx_k\right\|=0
\end{eqnarray*}
with $\|x_k\|\leq1$. This step completes the proof.
\end{proof}
Combining the related results in \cite{al2} and  \cite{fbmk2} with Theorem \ref{lt}, we have the following:
\begin{cor}\label{c6}Let $X$ be any normed space, $Y$ be a Banach space and $T_k\in B(X:Y)$ for all $k\in\mathbb{N}$.
Then, the following statements are equivalent:
\begin{enumerate}
 \item[(i)] The series $\sum_k T_k$ is $\ell_\infty(X)$-multiplier convergent.
\item[(ii)]  $\mathcal{S}:M_f^\infty(\sum_k T_k)\to Y$ is compact (weakly compact).
\item[(iii)]  $\mathcal{S}:M_{C}^\infty(\sum_kT_k)\to Y$ is compact (weakly compact).
\item[(iv)] $\mathcal{S}:M^\infty_{R}\big(\sum_k T_k\big)\to Y$ is compact (weakly compact).
    \end{enumerate}
\end{cor}

Now, we may introduce the multiplier space of weak $R$-convergence associated to the series $\sum_k T_k$ and obtain the corresponding results similar to the previous theorems and corollaries.
\begin{definition}
Let $X$ and $Y$ be normed spaces, and $T_k\in B(X:Y)$ for all $k\in\mathbb{N}$. The vector valued multiplier space $M^\infty_{wR}\big(\sum_k T_k\big)$ of weakly $R$-convergence associated to the series $\sum_k T_k$ is defined by
\begin{eqnarray*}
M^\infty_{wR}\big(\sum_k T_k\big):=\left\{x=(x_k)\in\ell_\infty(X):\sum_k T_k x_{k} \textrm{ is $wR$-convergent}\right\}
\end{eqnarray*}
and endowed with the sup norm.
\end{definition}
Since the inclusion $M^\infty_{R}\big(\sum_k T_k\big)\subseteq M^\infty_{wR}(\sum_k T_k)$ clearly holds, we have the following inclusions which are similar to the relation in (\ref{inc}):
  \begin{eqnarray}\label{inc1}\phi(X) \subseteq M^\infty_{R}\big(\sum_k T_k\big)\subseteq M^\infty_{wR}\big(\sum_k T_k\big)\subseteq \ell_\infty(X).
  \end{eqnarray}
 Since the completeness of the multiplier space $M_{wR}^\infty\big(\sum_{k}x_{k}\big)$ is given by means of $c_0(X)$-multiplier convergent series by using the similar technique for proving the completeness of $M_{R}^\infty\big(\sum_{k}x_{k}\big)$, we omit the proof of following theorem.
  \begin{thm}\label{t11} Let $X$ and $Y$ be any given Banach spaces, and $T_k\in B(X:Y)$ for all $k\in\mathbb{N}$. Then, the series $\sum_k T_k$ is $c_0(X)$-multiplier convergent if and only if
 $M^\infty_{wR}\big(\sum_k T_k\big)$ is a Banach space.
\end{thm}
\begin{cor}
Let $X$ and $Y$ be Banach spaces, and $T_k\in B(X:Y)$ for all $k\in\mathbb{N}$. Then, the series $\sum_kT_k$ is $c_0(X)$-multiplier convergent if and only if the inclusion $c_0(X)\subseteq M_{wR}^\infty\big(\sum_k T_k\big)$ holds.
\end{cor}
\begin{rem}
Let us suppose that $X$ and $Y$ are Banach spaces and $T_k\in B(X:Y)$ for all $k\in\mathbb{N}$. Then, the multiplier space $M_w^\infty\big(\sum_k T_k\big)$ is introduced in \cite{s1} as
 \begin{eqnarray*}
M_w^\infty\big(\sum_k T_k\big):=\left\{x=(x_k)\in\ell_\infty(X): \sum_kT_kx_k \textrm{ is weakly convergent}\right\}.
\end{eqnarray*} Now, if the series $\sum_kT_k$ is a $c_0(X)$-multiplier convergent,
 then the series $\sum_ky^*(T_kx_k)$ is convergent for all $x=(x_k)\in c_0(X)$ and for all $y^*\in Y^*$, that is, the series is weakly convergent. It is known by Corollary \ref{qw} that $x=(x_k)\in M_R^\infty\big(\sum_k T_k\big)$, and so $x=(x_k)\in M_{wR}^\infty\big(\sum_k T_k\big)$. This means that there exists $y_0\in Y$ with $wR-\sum_k T_kx_k=y_0$ such that
 \begin{eqnarray*}
\sum_ky^*(T_kx_k) =R-\sum_k y^*(T_kx_k)=y^*(y_0).
\end{eqnarray*}
Therefore, the inclusion $M_R^\infty\big(\sum_k T_k\big)\subseteq M_w^\infty(\sum_k T_k)$ holds. However, we have no an idea on the sufficient conditions for the reverse inclusion.
\end{rem}

By combining the previous results and Theorem \ref{t11}, we derive the following for the analogue of Corollary \ref{cor1} in the weak topology:
\begin{cor} Let $X$ and $Y$ be Banach spaces, and $T_k\in B(X:Y)$ for all $k\in\mathbb{N}$. Then, the following assertions are equivalent:
\begin{enumerate}
\item[(i)] The series $\sum_k T_k$ is $c_0(X)$-multiplier convergent.
\item[(ii)] $M^\infty_{wf}\big(\sum_k T_k\big)$ is a Banach space.
    \item[(iii)] $M^\infty_{wC}\big(\sum_k T_k\big)$ is a Banach space.
\item[(vi)] $M^\infty_{wR}\big(\sum_k T_k\big)$ is a Banach space.
\end{enumerate}
\end{cor}
Following theorem is the analogue of Theorem 3.3 of \cite{al2}. Since the proof is similar to the case $M^\infty_{wC}\big(\sum_k T_k\big)$, we omit details.
\begin{thm}\label{wy}
Let $X$ be a Banach space, $Y$ be any normed space and $T_k\in B(X:Y)$ for all $k\in\mathbb{N}$. Then, $Y$ is complete if and only if the multiplier space $M^\infty_{wR}\big(\sum_k T_k\big)$ is complete for every $c_0(X)$-multiplier Cauchy series.
\end{thm}
It is well-known that if $Y$ is a Banach space, then $B(X:Y)$ is a Banach space. So, Theorem \ref{wy} and also Theorem \ref{y} can be used for proving completeness of $B(X:Y)$.

Now, we give the following theorem which characterizes the continuity of weak summing operator with $c_0(X)$-multiplier Cauchy series.
\begin{thm}\label{t33}
Let $X$ and $Y$ be normed spaces, and $T_k\in B(X:Y)$ for all $k\in\mathbb{N}$. Then, the summing  operator $\mathcal{S}$ defined by
\begin{eqnarray}\label{sum1}
\begin{array}{ccccl}
\mathcal{S} & : &M^\infty_{wR}\big(\sum_k T_k\big)&\longrightarrow &Y \\
          &    &x=(x_k)&\longmapsto &\mathcal{S}x=wR-\sum_k T_k x_{k}.
\end{array}
\end{eqnarray} is continuous if and only if the series $\sum_{k}T_{k}$ is $c_0(X)$-multiplier Cauchy.
\end{thm}
\begin{proof}
Let us suppose that the summing operator $\mathcal{S}$ defined by (\ref{sum1}) is continuous and consider the set $G$ given by (\ref{eq213}). Then, the desired result follows from the inequality
\begin{eqnarray*}
\sup_{n\in\mathbb{N}}G=\left|wR-\sum_{k}T_{k}x_{k}\right|\leq \|\mathcal{S}\|,
\end{eqnarray*}
since the inclusion $\phi\subset M_{wR}^\infty\big(\sum_{k}T_{k}\big)$ holds.

Conversely, if $\sum_{k}T_{k}$ is $c_0(X)$-multiplier Cauchy series, then the set $G$ is bounded (see \cite[Theorem 1.3]{s1}) and so $H=\sup_{n\in\mathbb{N}}
G$. If $x=(x_{k})\in M^\infty_{wR}\big(\sum_k T_k\big)$, then the proof follows from the inequality
\begin{eqnarray*}
\|\mathcal{S}x\|=\left|R-\sum_k y^*\left(T_{k}x_{k}\right)\right|\leq H\|x\|
\end{eqnarray*}
for all $y^*\in B_{Y}$.
\end{proof}
From Theorem \ref{t33} and the conclusions due to Altay and Kama \cite{al2} and Karaku\c s and Ba\c sar \cite{fbmk2}, we have the following:
\begin{cor}\label{c5}
Let $X$ and $Y$ be two normed spaces, and $T_k\in B(X:Y)$ for all $k\in\mathbb{N}$. Then, the following statements are equivalent:
\begin{enumerate}
\item[(i)] The series $\sum_k T_k$ is $c_0(X)$-multiplier Cauchy.
\item[(ii)] $\mathcal{S}:M_{wf}^\infty\big(\sum_k T_k\big)\to Y$ is continuous.
\item[(iii)] $\mathcal{S}:M_{wC}^\infty\big(\sum_k T_k\big)\to Y$ is continuous.
\item[(iv)] $\mathcal{S}:M^\infty_{wR}\big(\sum_k T_k\big)\to Y$ is continuous.
\end{enumerate}
\end{cor}
\begin{thm}\label{lt1} Let $X$ be any normed space, $Y$ be a Banach space and $T_k\in B(X:Y)$ for all $k\in\mathbb{N}$. Then, the series $\sum_{k}T_{k}$ is $\ell_\infty(X)$-multiplier convergent if and only if the summing  operator $\mathcal{S}$ defined by (\ref{sum1}) is compact (weakly compact).
\end{thm}
\begin{proof}
Since the proof can be given by the similar way used in proving Theorem \ref{lt}, we omit details.
\end{proof}
By Theorem \ref{lt1} and the results in \cite{al2} and  \cite{fbmk2}, we have the following:
\begin{cor}
Let us suppose that $X$ and $Y$ are any normed spaces such that $Y$ is complete, and $T_k\in B(X:Y)$ for all $k\in\mathbb{N}$. Then the following statements are equivalent:
\begin{enumerate}
 \item[(i)] The series $\sum_k T_k$ is $\ell_\infty(X)$-multiplier convergent.
\item[(ii)]  $\mathcal{S}:M_{wf}^\infty(\sum_kT_k)\to Y$ is compact (weakly compact).
\item[(iii)]  $\mathcal{S}:M_{wC}^\infty(\sum_kT_k)\to Y$ is compact (weakly compact).
\item[(iv)] $\mathcal{S}:M^\infty_{wR}\big(\sum_k T_k\big)\to Y$ is compact (weakly compact).
    \end{enumerate}
\end{cor}

Prior to passing to the next section, we present Proposition \ref{eq} which states, besides the inclusion given in (\ref{inc1}), the iclusion $M^\infty_{wR}\big(\sum_k T_k\big)\subseteq M^\infty_{R}\big(\sum_k T_k\big)$ also holds.
\begin{pr}\label{eq}
Let $X$ and $ Y$ be normed spaces. If $\sum_k T_k$ is $\ell_\infty(X)$-multiplier Cauchy, then $M^\infty_{wR}\big(\sum_k T_k\big)= M^\infty_{R}\big(\sum_k T_k\big)$.
\end{pr}
\begin{proof}
Let $X$ and $ Y$ be normed spaces, and $x=(x_k)\in M^\infty_{wR}\big(\sum_k T_k\big)$. Then, there exists $y\in Y$ such that
$R-\sum_ky^*(T_kx_k)=y^*(y)$
for every $y^*\in Y^{*}$. From hypothesis, since the partial sums of the series $ \sum_{k}T_kx_k$ form a Cauchy sequence in $Y$, there exists $y^{**}\in Y^{**}$ such that
$R-\sum_kT_kx_k=y^{**}.$
If we consider the uniqueness of the limit, then we have $y^{**}=y$. Therefore, $x=(x_k)\in M^\infty_{R}\big(\sum_k T_k\big)$.
\end{proof}

\section{A version of Orlicz-Pettis Theorem for Riesz summability}
The Orlicz-Pettis Theorem is one of the important results in the Theory of Functional Analysis. Many generalizations and applications of this theorem can be found in \cite{al2,leo1,leo2,s2,s3,s4,yuro}.
Before stating and proving the Orlicz-Pettis Theorem by means of Riesz convergence, we will give the following definition (see \cite{s2}):
\begin{definition}
The space $\lambda$ has the infinite gliding hump property ($\infty$-GHP) if whenever $x\in \lambda$ and $\{\sigma_m\}$ is an increasing sequence of intervals, there exist a subsequence $\{p_m\}$ and $t_{p_{m}}>0$, $t_{p_{m}}\to \infty$ such that every subsequence of $\{p_m\}$ has a further subsequence $\{q_m\}$ such that the coordinatewise sum of the series $\sum_mt_{q_{m}}\chi_{\sigma_{q_{m}}}x\in \lambda$.
\end{definition}

In the following, we give a new version of Orlicz-Pettis theorem for vector valued multiplier Riesz convergent space of OVS by using Antosik-Mikusinski matrix theorem.

\begin{thm}
Let $\lambda$ have $\infty$-GHP and $T_k\in B(X:Y)$ for all $k\in\mathbb{N}$. If the series $\sum_kT_k$ is $\lambda-$multiplier Riesz convergent with respect to weak topology of $Y$, then the series $\sum_kT_k$ is $\lambda-$multiplier Riesz convergent with respect to strong topology of $Y$.
\end{thm}

\begin{proof}
Let $\epsilon>0$. If the conclusion is false, there exists $x\in \lambda$, $(y_n^*)\subset Y^*$ bounded sequence and an increasing sequence of intervals $\{\sigma_n\}$ such that
\begin{eqnarray}\label{eq4.1}
\left|R-\sum_{k\in \sigma_n}y_n^*(T_kx_k)\right|>\epsilon
\end{eqnarray}
for all $n\in \mathbb N$. Since $\lambda$ has $\infty$-GHP, there exist a subsequence $\{p_n\}$ and $t_{p_{n}}>0$, $t_{p_{n}}\to \infty$ such that every subsequence of $\{p_n\}$ has a further subsequence $\{q_n\}$ such that $\sum_nt_{q_{n}}\chi_{\sigma_{q_{n}}}x\in \lambda$.
Now, we consider the matrix $H=[h_{ij}]$ defined by
\begin{eqnarray*}
h_{ij}=\sum_{m\in \sigma_{p_{j}}}\frac{y^*_{p_{i}}}{t_{p_{i}}}(T_m(t_{p_{j}}x_m))
\end{eqnarray*}
for all $i,j\in\mathbb{N}$. Since $(y_i^*)$ is bounded and $t_{p_{i}}\to \infty$, the columns of $H$ converge to 0.  On the other hand, since $\lambda$ has $\infty$-GHP, we have $\sum_nt_{q_{j}}\chi_{\sigma_{q_{j}}}x\in \lambda$ for every subsequence $\{q_j\}$ and hence the subseries
\begin{eqnarray*}
\sum_j\sum_{m\in \sigma_{q_{j}}}T_m(t_{q_{j}}x_m)
\end{eqnarray*}
is weakly Riesz convergent. Then, we have
\begin{eqnarray*}
\lim_{i\to\infty}\sum_ja_{iq_{j}}=\lim_{i\to\infty}\sum_j\sum_{m\in \sigma_{q_{j}}}\frac{y^*_{p_{i}}}{t_{p_{i}}}(T_m(t_{q_{j}}x_m)).
\end{eqnarray*}
From Antosik-Mikusinski theorem (\cite{s2} Appendix D), the matrix $H$ is Riesz convergent to zero as diagonal, which is contradiction with (\ref{eq4.1}).
\end{proof}
\section{Conclusion}
P\'{e}rez-Fern\'{a}ndez et al. \cite{a5} gave some properties of a normed space $X$ like completeness and barrelledness through the $wuC$ series in $X$ and weakly* unconditionally Cauchy series in $X^*$. Aizpuru et al. \cite{a1} gave a new characterization of $wuC$ and $uc$ series through Ces\`{a}ro summability method and studied the new spaces associated to the series in a Banach space for proving completeness and barrelledness of normed spaces, and obtained a new version of Orlicz-Pettis theorem in scalar case. They also characterized the $wuC$ series by means of continuity of a linear mapping from these new spaces to a normed space. Aizpuru et al. \cite{a9} gave a new characterization of $wuC$ and $uc$ series via the completeness of some new subspaces of $\ell_\infty$ obtained from almost convergence and also presented a new version of Orlicz-Pettis theorem. In \cite{a4}, they also generalized the results given in \cite{a1} to any regular matrix. In recent times,
Karaku\c s and Ba\c sar \cite{fbmk1} introduced a slight generalization of almost convergence and gave some new multiplier spaces associated to the series $\sum_kx_k$ in a normed space by means of this new summability method. Swartz constructed some new versions of the Orlicz-Pettis theorem for multiplier convergent series under continuity assumptions on some linear operators, and gave applications to spaces of continuous linear operators, \cite{s3,s4}.  In \cite{yuro}, Yuanghong and Ronglu proved that multiplier convergence of OVS is closely related to the $AK$ property of the sequence spaces, and obtained corresponding versions of Orlicz-Pettis theorems. Kama and Altay \cite{al3} obtained some new multiplier space by using Fibonacci sequence spaces, and Kama et al. \cite{al4} gave similar results for multiplier spaces which were obtained from backward difference matrix. Most recently, Le\'{o}n-Saavedra et al. \cite{leo1,leo2} also obtained some new versions of Orlicz-Pettis theorem by using $w_p$-summability and a general summability methods.

Swartz \cite{s1} extended these studies (which are concerned with the scalar valued multiplier spaces) to the case of operator valued series and vector valued multipliers. Later, Altay and Kama \cite{al2}  introduced the vector valued multiplier spaces by means of Ces\`{a}ro convergence and a sequence of continuous linear operators.  Karaku\c s  gave the vector valued multiplier spaces $S_{\Lambda}(\mathcal{T})$ and $S_{w\Lambda}(\mathcal{T})$ by means of $\Lambda$-convergence and a series of bounded linear operators together with the characterization the completeness of both of the normed spaces $S_{\Lambda}(\mathcal{T})$ and $S_{w\Lambda}(\mathcal{T})$ via $c_0(X)$-multiplier convergence of an operator series, \cite{karakus1}. Quite recently, Karaku\c s and Ba\c sar \cite{fbmk2,fbmk4} and Kama \cite{kama} introduced the vector valued multiplier spaces of almost summability, a generalized almost summability and statistical Ces\`{a}ro summability, respectively.

It is worth noting to the reader that one can obtain corresponding results of the present paper by using any regular matrix $A=(\alpha_{nk})_{n,k\in\mathbb{N}}$ instead of the Riesz matrix.

\end{document}